\documentclass[11pt]{amsart}

\usepackage{hyperref}
\usepackage{amsmath} 
\usepackage{amsthm}
\usepackage{amssymb}
\usepackage{color}
\usepackage{tikz-cd}
\usepackage{adjustbox}
\usepackage[margin=1.5in]{geometry}
\usepackage{todonotes}
\usepackage{caption}
\usepackage{enumitem}
\usepackage{appendix}
\usepackage{float}
\usepackage{array}
\usepackage{longtable}
\usepackage{tocvsec2}
\usepackage{float}

\usepackage{listings}
\newtheorem*{rep@theorem}{\rep@title}
\newcommand{\newreptheorem}[2]{%
	\newenvironment{rep#1}[1]{%
		\def\rep@title{#2 \ref{##1}}%
		\begin{rep@theorem}}%
		{\end{rep@theorem}}}
\makeatother

\newtheorem{theorem}{Theorem}[section]
\newtheorem{lemma}[theorem]{Lemma}
\newtheorem{proposition}[theorem]{Proposition}

\newtheorem{corollary}[theorem]{Corollary}
\newtheorem{conjecture}[theorem]{Conjecture}

\theoremstyle{definition}

\newtheorem{remark}[theorem]{Remark}
\newreptheorem{theorem}{Theorem}
\newreptheorem{corollary}{Corollary}

\DeclareMathOperator{\C}{\mathcal{C}}
\DeclareMathOperator{\GC}{\mathcal{G}(\mathcal{C})}

\DeclareMathOperator{\Cpt}{\mathcal{C}_\mathrm{pt}}
\DeclareMathOperator{\Cad}{\mathcal{C}_\mathrm{ad}}

\DeclareMathOperator{\FPdim}{FPdim}

\DeclareMathOperator{\Rep}{Rep}
\DeclareMathOperator{\rank}{rank}

\newcommand\VecGw{\operatorname{\textbf{Vec}_G^{\omega}}}

\DeclareMathOperator{\fpdim}{FPdim}
\DeclareMathOperator{\svecg}{sVec}

\DeclareMathOperator{\ad}{ad}
\DeclareMathOperator{\pt}{pt}
\DeclareMathOperator{\vecg}{Vec}
\DeclareMathOperator{\sem}{semion}

\usepackage{fancyhdr}
\fancypagestyle{mypagestyle}{%
	\fancyhf{}
	\fancyhead[OC]{\sc \footnotesize MTCs with Frobenius-Perron dimension congruent to 2 modulo 4}
	\fancyhead[EC]{\sc \footnotesize A. Chakravarthy, A. Czenky, and J. Plavnik}
	\fancyhead[R]{\thepage}%
}
\title{On modular categories with Frobenius-Perron dimension congruent to 2 modulo 4}

\begin{document}

\author[A. Chakravarthy]{Akshaya Chakravarthy}
\address{Thomas Jefferson High School for Science and Technology, USA}
\email{2024achakrav@tjhsst.edu}

\author[A. Czenky]{Agustina Czenky}
\address{Department of Mathematics, University of Oregon, Eugene, OR 97403, USA}
\email{aczenky@uoregon.edu}

\author[J. Plavnik]{Julia Plavnik}
\address{\parbox{\linewidth}{Department of Mathematics, Indiana University}}
\email{jplavnik@iu.edu}

\begin{abstract}
We contribute to the classification of modular categories $\C$ with $\FPdim(\C)\equiv 2 \pmod 4$. We prove that such categories have group of invertibles of even order, and that they factorize as $\mathcal C\cong \widetilde{\mathcal C} \boxtimes \sem$, where $\widetilde{\mathcal C}$ is an odd-dimensional modular category and $\sem$ is the rank 2 pointed modular category. This reduces the classification of these categories to the classification of odd-dimensional modular categories. It follows that  modular categories $\C$ with  $\FPdim(\C)\equiv 2 \pmod 4$ of rank up to 46 are pointed. More generally, we prove that if $\C$ is a weakly integral MTC and $p$ is an odd prime dividing the order of the group of invertibles that has multiplicity one in $\fpdim(\C)$, then we have a factorization $\C \cong \tilde \C \boxtimes \vecg_{\mathbb Z_p}^{\chi},$  for $\tilde \C$ an MTC with dimension not divisible by $p$ and $\chi$ a non-degenerate quadratic form on $\mathbb Z_p$.

\end{abstract}

\maketitle

\tableofcontents

\section{Introduction}

Modular tensor categories (MTCs) have several applications in both mathematics and physics. For example, they are related to topological quantum field theory~\cite{T, MS}, topological quantum computation~\cite{R}, and  quantum groups~\cite{BK}. Besides, modular categories are of independent mathematical interest, and their classification and properties are currently actively studied.  

An MTC is a fusion category with braiding and ribbon structures, which satisfy a non-degeneracy condition. 
 One way to approach their classification is by rank, i.e., by the number of (isomorphism classes of) simple objects. It was proven in~\cite{BNRW1} that there are finitely many (up to equivalence) MTCs of a fixed rank, which makes this a more reasonable problem. It is currently being studied intensively, see for example  \cite{BNRW2, ABPP, BR, CGP, CP, HR, NRW, RSW}. MTCs in which all simple objects have integral Frobenius-Perron dimension are called \emph{integral}, and they are of particular interest since they correspond to categories of representations of ribbon factorizable finite-dimensional semisimple quasi-Hopf algebras. Recently,  a classification of integral MTCs up to rank 12 was reported in \cite{ABPP} and \cite{NRW}. Furthermore, MTCs of odd dimension (which are always integral) have been classified up to rank 23, see \cite{BR, CP, CGP}.
	  
	  \medbreak 
	  This paper is dedicated to the study of MTCs with Frobenius-Perron dimension congruent to 2 modulo 4, which are known to be integral by \cite[Lemma 4.4]{BPR} and \cite[Corollary 3.2]{DN}.  They  have a special advantage regarding classification:
	  it has been proven that they cannot be perfect \cite[Corollary 10.11]{BP}, that is, they have at least one non-trivial invertible object. Moreover, the square of the Frobenius-Perron dimension of any simple object divides the Frobenius-Perron dimension of the category \cite[Lemma 1.2]{EG}, and thus the Frobenius-Perron dimensions of all simple objects must be odd. Many of these properties resemble the ones that odd-dimensional modular categories enjoy but they have a major difference as well: they have nontrivial self-dual objects. Our results show that the similarities (and discrepancies) between these two classes of modular categories are not just a coincidence since they are closely connected.

   Our first main result shows that an MTC $\C$ with $\FPdim(\C)\equiv 2 \pmod 4$ and group of invertibles of even order factorizes in terms of an odd-dimensional MTC of lower rank. See the precise statement below. 

   \begin{reptheorem}{evenfac}
   Let $\mathcal C$ be an MTC with $\FPdim(\C)\equiv 2 \pmod 4$ and $|\GC|$ even. Then, the pointed subcategory of rank 2 is a modular subcategory of $\mathcal C$. Moreover, we have that $\mathcal C\cong \widetilde{\mathcal C} \boxtimes \sem$, where $\widetilde{\mathcal C}$ is an odd-dimensional modular category and $\sem$ is the rank 2 pointed modular category.
\end{reptheorem}

Using this result and equivariantization techniques, we can prove the following.
\begin{reptheorem}{gceven}
        Let $\mathcal C$ be an MTC with $\FPdim(\C)\equiv 2 \pmod 4$. Then, $|\mathcal G(\mathcal C)|$ is even.
    \end{reptheorem}

As a consequence of the two previous theorems we have the following.

   \begin{repcorollary}{class}
        Let $\mathcal C$ be an MTC with $\FPdim(\C)\equiv 2 \pmod 4$. Then, $\mathcal C\cong \widetilde{\mathcal C} \boxtimes \sem$, where $\widetilde{\mathcal C}$ is an odd-dimensional modular category and $\sem$ is the rank 2 pointed modular category.
    \end{repcorollary}

The upshot is that we can reduce the problem of classifying MTCs with Frobenius-Perron dimension congruent to 2 modulo 4 to classifying those of odd Frobenius-Perron dimension.
   
It has been proven that MTCs with dimension congruent to 2 modulo 4 up to rank 10 are pointed \cite[Theorem 1.2]{ABPP}. As a consequence of our results and the classification of low-rank odd-dimensional MTCs \cite{BR, CP, CGP}, we extend this to rank 46. That is, MTCs with Frobenius-Perron dimension congruent to 2 modulo 4 up to rank 46 are pointed and thus classified by group data.

\medbreak
We generalize the factorization result above to the case of $p$ an odd prime of multiplicity one in $\FPdim(\C),$ see the precise statement below. 

\begin{reptheorem}{factorization for p odd}
    Let $\mathcal C$ be a weakly integral MTC, and let $p$ be an odd prime dividing $|\GC|.$ Suppose that $p$ has multiplicity 1 in $\fpdim(\C).$ Then, $\C$ has a pointed modular subcategory of dimension $p$. In particular, $\C\cong \tilde \C \boxtimes \vecg_{\mathbb Z_p}^\chi$ for $\tilde \C$ an MTC of dimension not divisible by $p$ and $\chi$ a non-degenerate quadratic form on $\mathbb Z_p$.
\end{reptheorem}

More generally, we prove the following factorization result. 

\begin{repcorollary}{GC and GCad same prime factors}
    Let $\C$ be an MTC. Then 
    \begin{align*}
        \mathcal C \cong \mathcal P \boxtimes \tilde{\C},
    \end{align*}
    where $\mathcal P$ is a pointed modular category  and $\tilde{\C}$ is a modular category such that $\mathcal G(\tilde{\C})$ and $\mathcal G(\tilde{\C}_{\mathrm{ad}})$ have the same prime factors.  
\end{repcorollary}

	 This paper is organized as follows.
 A brief introduction to fusion and modular categories is given in Section~\ref{prelim}.  In Section \ref{factorization} we state and prove our main results regarding the factorization of MTCs. We also show the existence of pointed modular subcategories under certain conditions on the order of the group of invertibles. Lastly, in Section \ref{solvability}, we discuss the solvability of MTCs with dimension congruent to 2 modulo 4, and we prove a Feit-Thompson type result for the weakly group-theoretical case.

\settocdepth{part}
\section*{Acknowledgements}
	 We thank the MIT PRIMES-USA program and its coordinators Prof.~Pavel Etingof, Dr.~Slava Gerovitch, and Dr.~Tanya Khovanova for giving us this opportunity to do research. We also appreciate NSF grant DMS-2146392 and the Simons Foundation Award 889000 as part of the Simons Collaboration on Global Categorical Symmetries for partially supporting the research of J.P. The hospitality and excellent working conditions at the Department of Mathematics at the University of Hamburg are sincerely thanked by J.P., where she carried out part of this research as an Experienced Fellow of the Alexander von Humboldt Foundation. Lastly, we appreciate C. Galindo and C. Jones for helpful conversations.
\settocdepth{section}

\section{Preliminaries}\label{prelim}

We work over an algebraically closed field  $\mathbf k$ of characteristic zero. We let $\operatorname{Vec}$ be the category of finite dimensional vector spaces over $\mathbf k$ and $\Rep(G)$ be the category of finite dimensional representations of a finite group $G$.  We refer the reader to~\cite{ENO1, ENO2, EGNO} for the notions of tensor and fusion categories used throughout. 

\subsection{Fusion categories}
A \emph{fusion category} $\C$  is a rigid tensor category that is semisimple and has finitely many isomorphism classes of simple objects. 
 We denote by \textbf 1 its identity object, and by  $\mathcal{O(C)}$ the set of isomorphism classes of simple objects in $\mathcal C,$ so that $\rank(\C)=|\mathcal{O(C)}|$.

For an object $X$ in a fusion category $\mathcal C$, we denote by $\FPdim(X)$ its Frobenius-Perron dimension, see~\cite[Proposition 3.3.4]{EGNO}. The \emph{Frobenius-Perron dimension} $\FPdim(\mathcal C)$ of $\mathcal C$ is defined by
\begin{align*}
	\FPdim(\mathcal C): = \sum\limits_{X\in \mathcal{O(C)} } \FPdim(X)^2.
\end{align*}
The category $\C$ is said to be \emph{weakly integral} if $\FPdim(\mathcal C)$ is an integer, and \emph{integral} if $\FPdim(X)$ is an integer for all $X\in \C$.

Let $\mathcal C$ be a fusion category and $X$ an object in $\mathcal C$. If its evaluation $X^* \otimes X \to \mathbf 1$ and coevaluation $\mathbf 1 \to X \otimes X^*$ maps are isomorphisms, $X$ is said to be \emph{invertible}, see~\cite[Definition 2.11.1]{EGNO}. Equivalently, an object $X$ is invertible if it has Frobenius-Perron dimension one. Isomorphism classes of invertible objects in $\mathcal C$ as a group will be denoted as $\mathcal{G(C)}$. The \emph{pointed subcategory} $\mathcal{C}_\mathrm{pt}$ of $\C$
is the largest pointed subcategory of $\mathcal C$. We say that $\mathcal C$ is \emph{pointed} if all its simple objects are invertible, i.e, if $\C=\Cpt$. It is well known that any pointed fusion category $\mathcal C$ is equivalent to the category of finite dimensional $G$-graded vector
spaces $\VecGw$, where $G$ is a finite group and $\omega$ is a 3-cocycle on $G$ with coefficients in $\mathbf{k}^{\times}$. Hence pointed categories are classified by group data. 

For an object $X$ in a fusion category $\mathcal C$, let $X^*\in \mathcal C$ denote its dual. If $X\cong X^*$, we say $X$ is $\emph{self-dual}$. Otherwise, when $X\not\cong X^*$, we say that $X$ is \emph{non-self-dual}. Then $\C$ is said to be maximally-non-self-dual (MNSD) if $X\not\cong X^*$ for all simple $X\not\cong \mathbf 1$ in $\C$.
 
For a fusion category $\C$ and a finite group $G$, a \emph{$G$-grading} on $\mathcal C$ is a decomposition  $\mathcal C = \displaystyle\bigoplus_{g\in G} \mathcal C_g,$  such that  $\mathcal C_g$ is an abelian subcategory for all $g\in G$, and the grading behaves well under tensor product and duals, see~\cite[Section 4.14]{EGNO} and ~\cite[Section 2.1]{GN}. When $\mathcal C_g\ne 0$ for all $g \in G$, the grading is said to be \emph{faithful}. In such case, by~\cite[Proposition 8.20]{ENO1}, 
 all the components $\mathcal C_g$ have the same Frobenius-Perron dimension, and so $$\FPdim(\mathcal C)=|G|\FPdim(\mathcal C_e).$$
By~\cite{GN}, any fusion category $\mathcal{C}$ has a canonical faithful grading called the \emph{universal grading}. The trivial component in this grading coincides with the \emph{adjoint subcategory}  $\mathcal{C}_{\mathrm{ad}}$ of $\mathcal C$, which is the fusion subcategory generated by $X\otimes X^*$ for all $X \in \mathcal{O(C)}$.  Let $\mathcal{U(C)}$ be the group corresponding to the universal grading of $\C$. We then have that if $\mathcal{C}$ is equipped with a braiding, then $\mathcal{U(C)}$ is abelian. Also, if $\mathcal C$ is modular, then $\mathcal{U(C)}$  is isomorphic to $\mathcal{G(C)}$~\cite[Theorem 6.3]{GN}.

\subsubsection{Solvable fusion categories}

A fusion category $\mathcal{C}$ is \emph{solvable} if it is Morita equivalent to a cyclically nilpotent fusion category, see~\cite{ENO2}. The class of solvable fusion categories is closed under the Deligne tensor product, Drinfeld centers, fusion subcategories, and extensions and equivariantizations by solvable groups,~\cite[Propositions 4.1 and 4.5]{ENO2}.

\subsection{Modular tensor categories}

A \emph{braiding} on a fusion category  $\mathcal C$ is a natural isomorphism $$c_{X,Y}: X\otimes Y \xrightarrow{\cong} Y \otimes X,$$ for all $X,Y\in \C$, satisfying the hexagon axioms, see~\cite[Definition 8.1.1]{EGNO}. Then $\mathcal C$ is said to be \emph{braided} if it is equipped with a braiding. 
 Let $\mathcal{C}$ be a braided fusion category. A \emph{pivotal structure} on $\mathcal C$ is a natural isomorphism $\psi: \text{Id} \xrightarrow{\sim} (-)^{**}$,  see~\cite{BW, EGNO}.  Associated to a pivotal structure is the left and right trace of a morphism $X\to X$, see e.g.~\cite[Section 4.7]{EGNO}. The pivotal structure is called \emph{spherical} if for any such morphism both traces are equal.

A \emph{pre-modular} tensor category is a braided fusion category equipped with a spherical structure ~\cite[Definition 8.13.1]{EGNO}. 
 The \emph{S-matrix} $S$ of a pre-modular category $\mathcal C$ with braiding $c_{X,Y}: X\otimes Y \xrightarrow{\cong} Y \otimes X$ is defined by $S:= \left(s_{X,Y}\right)_{X,Y \in \mathcal{O(\mathcal{C})}}$, where $s_{X,Y}$ is the trace of $c_{Y,X}c_{X,Y}:X\otimes Y \to X\otimes Y$, see \cite[Section 8.13]{EGNO}. When the $S$-matrix is invertible, a pre-modular tensor category $\mathcal C$ is said to be \emph{modular}.

In this work, the terms ``premodular tensor category'' and ``modular tensor category'' imply semisimplicity, in order to be consistent with prior papers. Also, the term ``modular category'' is equivalent to a ``modular tensor category'' as defined in this paper.

Let $\mathcal{C}$ be a braided fusion category  with braiding $c_{X,Y}:X\otimes Y\xrightarrow{\cong}Y\otimes X$.  The \emph{M\"{u}ger centralizer} of a fusion subcategory $\mathcal{K}$ is the fusion subcategory $\mathcal K'$ of $\mathcal C$ given by the objects $Y$ in $\mathcal{C}$ satisfying
\begin{align}\label{centralizador}
	c_{Y,X} \circ c_{X,Y}=\operatorname{id}_{X\otimes Y}, \ \ \text{for all}\ X \in \mathcal{K},
\end{align} 
see~\cite{M}. We say  $\mathcal K$  is a \emph{modular subcategory}  if and only if $\mathcal K \cap \mathcal K' = \vecg$.
When  $\mathcal C'=\mathcal C$, we say that $\C$ is symmetric. A symmetric fusion category is called \emph{Tannakian} if there exists a finite group $G$ such that it is equivalent as a braided fusion category to $\Rep(G)$. In particular, symmetric fusion categories of odd Frobenius-Perron dimension are Tannakian~\cite[Corollary 2.50]{DGNO1}.

\begin{remark}\cite[Remark 2.2]{CP}
	If $\mathcal{C}$ is a braided fusion category, then $(\mathcal{C}_{\mathrm{ad}})_{\mathrm{pt}}$  is symmetric. 
\end{remark}

\subsubsection{Equivariantization and de-equivariantization}\label{section: equivariantization}
Let $\C$ be a braided fusion category with an action of a finite group $G$, see~\cite[Definition 4.15.1]{EGNO}. The \emph{equivariantization} of $\mathcal C$ by $G$ is the 
braided fusion category $\mathcal C^G$ of $G$-equivariant objects in $\mathcal C$, see \cite[Section 2.7]{EGNO}. 
Conversely, for $\mathcal C$ a braided fusion category with a Tannkian subcategory $\Rep(G)$, we can construct the \emph{de-equivariantization} $\mathcal C_G$ of $\mathcal C$ with
respect to $\Rep(G)$, which is also a braided fusion category, see~\cite[Theorem 8.23.3]{EGNO} for the construction. It turns out that there is an equivalence of braided fusion categories $\mathcal C\cong (\mathcal C_G)^G$.
By \cite[Proposition 4.26]{DGNO1}, these constructions satisfy $\FPdim(\mathcal C^G)=|G|\cdot \FPdim(\mathcal C)$  and  $\FPdim(\mathcal C_G)=\frac{1}{|G|}\cdot \FPdim(\mathcal C).$

If $\mathcal C$ is an MTC such that $(\mathcal{C}_{\mathrm{ad}})_{\mathrm{pt}}$ is odd-dimensional, then it follows from~\cite[Remark 2.2]{CP} that  $(\mathcal{C}_{\mathrm{ad}})_{\mathrm{pt}}$ is symmetric and thus Tannakian~\cite[Corollary 2.50]{DGNO1}, so it is equivalent to $\Rep(G)$ for some finite group $G$. Then by \cite[Remark 2.3]{ENO2}, $(\mathcal C_{\mathrm{ad}})_G$ is modular.

\section{Factorization} \label{factorization}
In this section, we  show the existence of modular subcategories for certain types of MTCs, and prove our factorization results.

\subsection{MTCs of dimension congruent to 2 modulo 4}
We prove here one of our main results, which states that if $\mathcal C$ is an MTC with $\FPdim(\C)\equiv 2 \pmod 4$, then $\C$ factorizes as $\mathcal C\cong \widetilde{\mathcal C} \boxtimes \sem$, where $\widetilde{\mathcal C}$ is an odd-dimensional MTC and $\sem$ is the rank 2 pointed modular category. We will need the following lemmas.

\begin{lemma}\label{lemma:casesinv}
Let $\mathcal C$ be an MTC with $\FPdim(\C)\equiv 2 \pmod 4.$ Then, exactly one of the following is true,
\begin{enumerate}
    \item $|\mathcal{G(C)}| \equiv 2 \pmod 4$ and $\fpdim(\mathcal C_{\ad})$ is odd or
    \item $|\mathcal{G(C)}|$ is odd and $\fpdim(\mathcal C_{\ad}) \equiv 2 \pmod 4$.
\end{enumerate}
\end{lemma}
\begin{proof}
The statement follows from taking equivalence modulo 4 on the equality  $\fpdim(\mathcal C) = |\mathcal{G(C)}|\fpdim(\mathcal C_{\textrm{ad}}).$  
\end{proof}

\begin{corollary}\label{odd rank group adjoint}
Let $\mathcal C$ be an MTC with $\FPdim(\C)\equiv 2 \pmod 4.$ Then $\mathcal G(\mathcal C_{\ad})$ has odd order.
\end{corollary}
\begin{proof}
 From Lemma~\ref{lemma:casesinv}, $|\mathcal{G(C)}|$ is either odd or congruent to 2 modulo 4. If the former is true, the statement follows trivially. In the latter case, we have that $\fpdim(\mathcal C_{\textrm{ad}})$ must be odd from Lemma~\ref{lemma:casesinv}, and so $|\mathcal G(\mathcal C_\textrm{ad})|$ is odd. 
\end{proof}

\begin{lemma} \label{svecmod}
    Let $\mathcal C$ be an MTC with $\FPdim(\C)\equiv 2 \pmod 4$ and $|\GC|$ even, and let $f$ be the invertible  of order 2 in $\mathcal C$. Then exactly one of the following holds.
    \begin{itemize}
        \item The fusion subcategory $\mathcal C[f]$ is modular, or
        \item The fusion subcategory $\mathcal C[f]$ is equivalent to $\svecg$ and $f$ is a fermion.
    \end{itemize}
\end{lemma}
\begin{proof}
 Consider the fusion subcategory $\mathcal C[f]$ generated by $f$. If $\mathcal C[f] \cap \mathcal C[f]' = \vecg$, then $\mathcal C[f]$ is modular and the claim holds in this case. 

On the other hand, suppose $\mathcal C[f] \cap \mathcal C[f]'$ is non-trivial. Then, since the order of $f$ is 2, we have that $\mathcal C[f] \cap \mathcal C[f]' = \mathcal C[f]$. Thus, 
$\mathcal C[f]$ is a symmetric subcategory of $\mathcal C$. Suppose first that $\mathcal C[f]$ is Tannakian and is equivalent to $\Rep(\mathbb Z_2)$. Then, from \cite[Section 4]{N2}, we would have that $4$ divides $\fpdim(\mathcal C)$, which is not possible. So, $\mathcal C[f]$ must be equivalent to $\svecg$, and we then have that $f$ is a fermion \cite[Definition 2.1]{BGHNPRW}. So, the lemma is proven.
\end{proof}

\begin{lemma} \label{c}
 Let $\mathcal C$ be an MTC with $\FPdim(\C)\equiv 2 \pmod 4$ such that  $\mathcal{G(C)} \cong \mathbb Z_2$. Then $\mathcal C_{\ad}$ is an odd-dimensional MTC and we have a decomposition $\mathcal C\cong \mathcal C_{\ad}  \boxtimes \sem,$ where $\sem$ is the rank 2 pointed modular category.  
\end{lemma}
\begin{proof}
 By Corollary~\ref{odd rank group adjoint}, $\mathcal G(\mathcal C_{\textrm{ad}})$ is an odd order subgroup of $\mathcal{G(C)} \cong \mathbb Z_2$, and so it must be trivial. By the proof of \cite[Lemma 5.2]{CP}, this implies that $\mathcal C_{\ad}$ and $\mathcal C_{\pt}$ are modular, and we have the factorization $\mathcal C\cong\mathcal C_{\ad}  \boxtimes  \mathcal C_{\pt}$. Lastly, by Lemma \ref{lemma:casesinv}, $\Cad$ is odd-dimensional.
\end{proof}

Recall that if $|\mathcal G(\mathcal C)| = 2$, then $\mathcal C[f]$ is modular, see Lemma~\ref{c}. 
This is also the case in the more general setting when the group of invertibles has even order. 
In the following theorem, we show that when the group of invertibles has even order, the category can always be determined in terms of an odd-dimensional modular category. 

\begin{theorem}\label{evenfac}
Let $\mathcal C$ be an MTC with $\FPdim(\C)\equiv 2 \pmod 4$ and $|\GC|$ even. Then, the pointed subcategory of rank 2 is a modular subcategory of $\mathcal C$. Moreover, we have that $\mathcal C\cong \widetilde{\mathcal C} \boxtimes \sem$, where $\widetilde{\mathcal C}$ is an odd-dimensional modular category and $\sem$ is the rank 2 pointed modular category.
\end{theorem}
\begin{proof}
 Consider first when $|\mathcal G(\mathcal C)| = 2$. 
Then, by Lemma~\ref{c}, we have that $\mathcal C \cong \mathcal C_{\ad} \boxtimes \mathcal C_{\pt}$, with $\mathcal C_{\pt}$ the rank 2 pointed modular category, that is, $\mathcal C_{\pt} \cong \sem$. Moreover, $\mathcal C_{\ad}$ is an odd-dimensional modular category in this case. 
    
Now, consider when $|\mathcal G(\mathcal C)| > 2$. From Lemma~\ref{svecmod}, there is an invertible $f$ of order 2 and $\mathcal C[f]$ is either modular or equivalent to $\svecg$. If it is modular, then the claim holds so suppose the latter. 
    
From \cite[Proposition 2.4]{BGHNPRW}, $f$ is a fermion, and there is a faithful $\mathbb Z_2$-grading of $\mathcal C\cong \mathcal C_0 \oplus \mathcal C_1$ such that a simple object $X \in \mathcal C_0$ if $\epsilon_X = 1$ and $X \in \mathcal C_1$ if $\epsilon_X = -1$. For the definition of $\epsilon_X$, see \cite[Proposition 2.3(ii)]{BGHNPRW}.

Suppose first that $\mathcal C_1$ contains at least one invertible object $h$. In a similar argument as in \cite[Proposition 3.17]{CGP}, we have that $\mathcal C_0$ and $\mathcal C_1$ contain the same number of invertibles. Suppose the order of $h \in \mathcal C_1$ is $n$ and so $h^n \in \mathcal C_0$ and $n$ must be even, which follows a similar argument as the one in the proof of \cite[Theorem 4.1(i)]{N1}. Therefore, the order of all invertibles in $\mathcal C_1$ must be even.
Since $\mathcal C_0$ and $\mathcal C_1$ have the same number of invertible objects, all invertible objects of even order should be in $\mathcal C_1$. However, from \cite[Proposition 2.3(ii)]{BGHNPRW}, we have that $f \in \mathcal C_0$, and so, this is impossible. 

We then have that all invertible objects are in $\mathcal C_0$, and that $\epsilon_h = 1$ for all invertible objects $h$ in $\mathcal C$. So, from \cite[Proposition 2.3(ii)]{BGHNPRW}, we have that 
    \[s_{f,h} = \epsilon_h\fpdim(h) = 1\] for all invertibles $h$.

So, $f \in (\mathcal C_{\pt})' \cong \mathcal C_{\ad}$. However, recall from Lemma~\ref{odd rank group adjoint} that $|\mathcal G(\mathcal C_{\ad})|$ has odd order. Thus, a contradiction is derived in this case and the theorem is proven.
\end{proof}

\begin{remark}From \cite[Corollary 8.2]{NS}, $\tilde \C$ in the theorem above is MNSD, i.e. all non-trivial simple objects in $\tilde \C$ are non-self dual.
    Hence, if $\mathcal C$ is an MTC of dimension congruent to 2 modulo 4, it follows from the previous theorem that  $\mathcal C$ has exactly two self-dual objects, namely, the unit and the invertible of order 2.
\end{remark}

\begin{theorem} \label{de-equivodd}
    Let $\mathcal C$ be an MTC with $\fpdim(\C)\equiv 2 \pmod 4$, and suppose that $(\mathcal C_{\ad})_{\pt} \cong \Rep(G)$ for a non-trivial odd-order group $G$. Then, $(\mathcal C_{\ad})_G$ contains an odd number of invertibles.     
\end{theorem}

\begin{proof}
    If $\mathcal C_{\ad}$ is odd-dimensional, the claim immediately follows. So, suppose that $\mathcal C_{\ad}$ has dimension congruent to 2 modulo 4. Then, $(\mathcal C_{\ad})_G$ is an MTC of dimension congruent to 2 modulo 4, as well. For the sake of contradiction, suppose there are an even number of invertibles in $\mathcal D := (\mathcal C_{\ad})_G$. In particular, we have that $|\mathcal G(\mathcal D)| \equiv 2 \pmod{4}$. Thus, $\mathcal D$ has exactly one invertible $h$ of order 2. Moreover, it follows from Theorem \ref{evenfac} that $\mathcal D \cong \operatorname{semion} \boxtimes \widetilde{D}$, for $\widetilde{D}$ an odd-dimensional MTC and $\sem$ the rank 2 pointed modular category. Notice that the action of $G$ on the $\operatorname{semion}$ is trivial on objects (and morphisms) since the order of $G$ is odd. 

    Now, we can consider the exact sequence given in \cite[Remark 3.1]{BN} that describes the invertibles in the $G$-equivariantization. Since $G$ is an abelian group, the exact sequence in this case is given by
    \[1\to \widehat{G} \to  \mathcal G (\mathcal C_{\ad}) \to G_0(\mathcal D)\to 1,\]
where $\hat G$ denotes the group of invertible characters of $G$ and $G_0(\mathcal D)$ 
the subgroup of $\mathcal G(\mathcal D)$ consisting of isomorphism classes of $G$-equivariant invertible objects. 
Furthermore, since $G\cong \mathcal G(\mathcal C_{\ad})$ is abelian and the sequence above is exact, it follows that $G_0(\mathcal D)$ should be trivial, which contradicts that the action of $G$ on the semion is trivial.
\end{proof}

   \begin{theorem}\label{gceven}
        Let $\mathcal C$ be an MTC with $\FPdim(\C)\equiv 2 \pmod 4$. Then, $|\mathcal G(\mathcal C)|$ is even.
    \end{theorem}
    
    \begin{proof}
        Suppose, for the sake of contradiction, that there is an MTC with dimension congruent to 2 modulo 4 and an odd number of invertibles, and let $\mathcal C$ be such an MTC of the smallest dimension. Notice that $\mathcal C$ must have dimension greater than 2 because if not, then $\mathcal C$ would be pointed, which would imply that $|\mathcal G(\mathcal C)|$ is even. So, we must have that $\fpdim(\mathcal C) \geq 6$. 
        
        If $|\mathcal G(\mathcal C_{\ad})| = 1$, then recall by the proof of \cite[Lemma 5.2]{CP} that $\mathcal C_{\ad}$ is an MTC and we have a factorization $\mathcal C \cong \mathcal C_{\ad} \boxtimes \mathcal C_{\pt}$. By assumption, $\mathcal C_{\pt}$ is odd-dimensional, and so $\mathcal C_{\ad}$ is an MTC of dimension congruent to 2 modulo 4. This contradicts that $|\mathcal G(\mathcal C_{\ad})| = 1$ since MTCs of dimension congruent to 2 modulo 4 have at least one non-trivial invertible object, see \cite[Corollary 10.11]{BP}.

        Hence, we should have that $|\mathcal G(\mathcal C_{\ad})| > 1.$ Consider the de-equivariantization of $\mathcal C_{\ad}$ by $G:=\mathcal G(\mathcal C_{\ad})$. Since $|\mathcal G(\mathcal C)|$ is odd, we must  have that $\fpdim(\mathcal C_{\ad}) \equiv 2 \pmod{4}$ from Lemma~\ref{lemma:casesinv}. Notice that $(\mathcal C_{\ad})_G$ is an MTC with  $\fpdim((\mathcal C_{\ad})_G) \equiv 2 \pmod{4}$, and since $G$ is not trivial, we get that  $\fpdim((\mathcal C_{\ad})_G) < \fpdim(\mathcal C)$. Again, since by assumption $\mathcal C$ is an MTC of the smallest dimension with $|\mathcal{G(C)}|$ odd, $(\mathcal C_{\ad})_{G}$ must have an even number of invertibles, but this is not possible from Theorem~\ref{de-equivodd}.

        So, neither cases are possible, and the theorem is proven.
    \end{proof}

    \begin{corollary} \label{class}
        Let $\mathcal C$ be an MTC with $\FPdim(\C)\equiv 2 \pmod 4$. Then, $\mathcal C\cong \widetilde{\mathcal C} \boxtimes \sem$, where $\widetilde{\mathcal C}$ is an odd-dimensional modular category and $\sem$ is the rank 2 pointed modular category.
    \end{corollary}
    \begin{proof}
        From Theorem~\ref{gceven}, we must have that $|\mathcal G(\mathcal C)|$ is even, and the claim follows directly from Theorem~\ref{evenfac}.
    \end{proof}

\begin{corollary}
    MTCs with Frobenius-Perron dimension congruent to 2 modulo 4 and rank up to 46 are pointed.
\end{corollary}
\begin{proof}
    This follows from Corollary \ref{class} and the classification of odd-dimensional MTCs up to rank 23, which are all pointed, see \cite{BR, CP, CGP}.
\end{proof}

\begin{remark}\label{rem: mod or tan}
    Let $\mathcal C$ be an MTC of dimension congruent to 2 modulo 4 with $|\mathcal G(\mathcal C)| > 2$. Let $g$ be an invertible of order $p$, for $p$ an odd prime dividing $|\mathcal G(\mathcal C)|$. Then the fusion subcategory $\mathcal C[g]$ is either modular or Tannakian. 
\end{remark} 

\begin{proof}
    If $p$ is a prime dividing $|\mathcal G(\mathcal C)|$, by Cauchy's Theorem, there is an invertible $g$ of order $p$. Since  $|\mathcal G(\mathcal C)| > 2$, we let $p$ be an odd prime. Consider the fusion subcategory $\mathcal C[g]$ generated by $g$. Since the order of $g$ is prime then   $\mathcal C[g] \cap \mathcal C[g]' = \vecg$ or $\mathcal C[g]$. In the first case, $\mathcal C[g]$ is modular. On the other hand, if $\mathcal C[g] \cap \mathcal C[g]' = \mathcal C[g]$, then $\mathcal C[g]$ is an odd-dimensional symmetric subcategory of $\mathcal C$, and must be Tannakian~\cite[Corollary 2.50]{DGNO1}.
\end{proof}

\subsection{Factorization for $p$ odd}

In this subsection, we generalize the factorization result for the case congruent to 2 modulo 4 to any weakly integral MTC and $p$ an odd prime that divides the group of invertibles and has multiplicity one in $\FPdim(\C).$ We also show the existence of pointed modular subcategories under certain conditions on the order of the group of invertibles.

\begin{theorem}\label{factorization for p odd}
    Let $\mathcal C$ be a weakly integral MTC, and $p$ an odd prime dividing $|\GC|.$ Suppose that $p$ has multiplicity one in $\fpdim(\C).$ Then, $\C$ has a pointed modular subcategory of dimension $p$. In particular, $\C\cong \tilde \C \boxtimes \vecg_{\mathbb Z_p}^\chi$ for $\tilde \C$ an MTC of dimension not divisible by $p$ and $\chi$ a non-degenerate quadratic form on $\mathbb Z_p$.
\end{theorem}
\begin{proof}
 Let $g$ be an invertible of order $p$. Consider the fusion subcategory $\mathcal C[g]$ generated by $g$. If $\mathcal C[g] \cap \mathcal C[g]' = \vecg$, then $\mathcal C[g]$ is modular and the claim holds. Suppose then that $\mathcal C[g] \cap \mathcal C[g]'$ is non-trivial. Since the order of $g$ is prime, we have that $\mathcal C[g] \cap \mathcal C[g]' = \mathcal C[g]$. Thus, 
$\mathcal C[g]$ is a symmetric subcategory of $\mathcal C$, and hence Tannakian since $p$ is odd \cite[Corollary 2.50]{DGNO1}. But then from \cite[Section 4]{N2} we would have that $p^2$ divides $\fpdim(\mathcal C)$, which is a contradiction. So, $\mathcal C[g]$ must be equivalent to $\vecg$, and $\C[g]$ is modular, as desired. 
\end{proof}

\begin{theorem} \label{theo: rel prime}
    Let $\mathcal C$ be an MTC and $G$ be a subgroup of $\GC$ such that $\mathcal G(\mathcal C_{\ad}) \leq G  \leq \GC$ and $|G|$  and $\frac{|\mathcal G(\C)|}{|G|}$ are relatively prime. Then, $\mathcal C$ contains a modular pointed subcategory of dimension $[\mathcal G(\mathcal C):G]$.
\end{theorem}
\begin{proof}
 Consider the fusion subcategory
\[\mathcal D := \bigoplus_{g\in G} \C_g.\]
Notice that since $\mathcal C_{\ad} \subset \mathcal D$, we have that $\mathcal D' \subset (\mathcal C_{\ad})^{'} \cong \mathcal C_{\pt}$ \cite[Corollary 6.9]{GN}. Since $\fpdim(\mathcal D) = |G|\fpdim(\Cad)$, it follows from \cite[Theorem 3.2]{M}, that $\fpdim(\mathcal D') = \frac{|\mathcal G(\C)|}{|G|}$. So, by assumption, all objects (which are invertibles) in $\mathcal D'$ have order co-prime with $|G|$. Then, since $\mathcal G(\mathcal C_{\ad}) \leq G$, we have that all objects in $\mathcal D'$ have order co-prime with $|\mathcal G(\mathcal C_{\ad})|$, and thus, $\mathcal D' \cap \mathcal C_{\ad} = \vecg$. Let $t$ be a simple (then invertible) object in $\mathcal D \cap \mathcal D'$. Then $t$ is the unit or has non-trivial order dividing $ \frac{|\mathcal G(\C)|}{|G|}$. Since $\mathcal D$ is graded by $G$, it follows from \cite[Proposition 4.1(i)]{N1} that $t$ must be in $\mathcal C_{\ad}$, but this implies that $t$ is the unit. Thus, $\mathcal D\cap \mathcal D' = \vecg,$ and so $\mathcal D'$ is a modular pointed category of dimension $[\mathcal G(\mathcal C):G]$, as desired.
\end{proof}

\begin{remark}
    Let $\mathcal C$ be an MTC, and let $p$ be a prime that divides $|\GC|$ but not $|\mathcal{G(\Cad)}|$. It follows from the theorem above that $\mathcal C$ contains a modular pointed subcategory of dimension $p^n$, where $p^n$ is the highest power of $p$ that divides $|\GC|$. In fact, if $|\GC|=p^nm,$ take $G$ as the subgroup of $\GC$ of order $m$. 
\end{remark}

\begin{corollary}\label{GC and GCad same prime factors}
    Let $\C$ be an MTC. Then 
    \begin{align*}
        \mathcal C \cong \mathcal P \boxtimes \tilde{\C},
    \end{align*}
    where $\mathcal P$ is a pointed modular category  and $\tilde{\C}$ is a modular category such that $\mathcal G(\tilde{\C})$ and $\mathcal G(\tilde{\C}_{\mathrm{ad}})$ have the same prime factors.  
\end{corollary}

\begin{remark}
       Let $\mathcal C$ be an MTC such that $|\mathcal G(\Cad)|$  and $\frac{|\mathcal G(\C)|}{|\mathcal G(\Cad)|}$ are relatively prime. It follows from the corollary above that
       \begin{align*}
        \mathcal C \cong \mathcal P \boxtimes \tilde{\C},
    \end{align*}
    where $\mathcal P$ is a pointed modular category of dimension $\frac{|\mathcal G(\C)|}{|\mathcal G(\Cad)|}$, and $\tilde{\C}$ is a modular category satisfying that $\mathcal G(\tilde{\C}) = \mathcal G(\tilde{\C}_{\mathrm{ad}})$. In particular, this is true for the case where $|\mathcal G(\mathcal C)|$ is square-free.
\end{remark}

\begin{proposition}\label{prop: pointed product of cyclic}
Let $\mathcal C$ be a weakly-integral MTC and suppose that $|\mathcal G(\mathcal C)| = p^2$ for an odd prime $p$. If $\mathcal G(\mathcal C) \cong \mathbb Z_p \times \mathbb Z_p$, then $\mathcal C$ contains a pointed modular subcategory of dimension $[\mathcal G(\mathcal C): \mathcal G(\mathcal C_{\ad})]$.
\end{proposition}
\begin{proof}
If $|\mathcal G(\mathcal C_{\ad})| = 1$, then from the proof of \cite[Lemma 5.2]{CP}, $\mathcal C \cong \mathcal C_{\ad} \boxtimes \mathcal C_{\pt}$, and since $\fpdim(\mathcal C_{\pt}) = [\mathcal G(\mathcal C):\mathcal G(\mathcal C_{\ad})]$, the claim holds in this case. 

So, suppose that $|\mathcal G(\mathcal C_{\ad})| > 1$, and thus, we have that $|\mathcal G(\mathcal C_{\ad})| = p$. Then, from \cite[Proposition 3.17]{CGP}, we have exactly $p$ components that contain invertible objects. More precisely, each of them contains $p$ invertibles.
Since $\mathcal G(\mathcal C) \cong \mathbb Z_p \times \mathbb Z_p$, we have that there are $p+1$ distinct subgroups of order $p$. Then, there must be a fusion subcategory 
\[\mathcal D := \mathcal C_{\ad} \oplus \mathcal C_{g} \oplus \mathcal C_{g^2} \oplus \cdots \oplus \mathcal C_{g^{p-1}},\]
with $\mathcal G(\mathcal D) \cong \mathcal G(\mathcal C_{\ad})$ and $g$ of order $p$.
Since $\mathcal C_{\ad} \subset \mathcal D$, we have that $\mathcal D' \subset (\mathcal C_{\ad})^{'} \cong \mathcal C_{\pt}$ \cite[Corollary 6.9]{GN}. From \cite[Theorem 3.2]{M}, it follows that $\fpdim(\mathcal D) = \frac{\fpdim(\mathcal C)}{p}$ and then $\fpdim(\mathcal D') = p$. If $\mathcal D \cap \mathcal D' = \vecg$, then the claim holds. 

So, suppose that $\mathcal D \cap \mathcal D'$ is not trivial, and we have $\mathcal D \cap \mathcal D' = \mathcal D'$. So $\mathcal G(\mathcal D') \cong \mathcal G(\mathcal D) \cong \mathcal G(\mathcal C_{\ad})$. Then, notice from \cite[Proposition 9.6.6]{EGNO} that we can use \cite[Theorem 8.21.1(i)]{EGNO} to obtain
     \[\fpdim(\mathcal C_{\ad} \cap \mathcal D')\fpdim(\mathcal D) = \fpdim(\mathcal D \cap (\mathcal C_{\ad})')\fpdim(\mathcal C_{\ad}).\] 
Since $\fpdim(\mathcal C_{\ad} \cap \mathcal D') = p = \fpdim(\mathcal D \cap (\mathcal C_{\ad})')$ and $\fpdim(\mathcal D) = \frac{\fpdim(\mathcal C)}{p} = p \fpdim(\mathcal C_{\ad})$, we get a contradiction  and the proposition is proven.
\end{proof}

\section{Solvability} \label{solvability}

Recall that it was conjectured in \cite[Conjecture 1.2]{CP} that any odd-dimensional modular category is solvable. The corollary bellow follows directly from Corollary \ref{class} and \cite[Proposition 4.5]{ENO2}.

   \begin{corollary}\label{cor: solvable iif}
        Odd-dimensional MTCs are solvable if and only if MTCs of dimension congruent to 2 modulo 4 are solvable.
    \end{corollary}

Hence \cite[Conjecture 1.2]{CP} is equivalent to the following. 
\begin{conjecture} \label{solvable}
MTCs of Frobenious-Perron dimension not divisible by 4 are solvable.
\end{conjecture}
    
Feit-Thompson's Theorem states that groups of odd order are solvable, and it can be generalized to the case of groups of order not divisible by 4. We will use this result in the following lemma that generalizes \cite[Proposition 7.1]{NP}, which is a Feit-Thompson Theorem for weakly group-theoretical fusion categories. Our proof is similar 
to that of \cite[Proposition 7.1]{NP} 
but we include it for completeness.

 \begin{lemma}\label{lemma: wgt dim cong 2 solv}
    Let $\mathcal C$ be a weakly group-theoretical fusion category with dimension congruent to 2 modulo 4. Then, $\mathcal C$ is solvable.
\end{lemma}
\begin{proof}
Since $\mathcal C$ is a weakly group-theoretical fusion category, $\mathcal C$ is Morita equivalent to a nilpotent fusion category. By \cite[Proposition 4.5]{ENO2}, it is enough to show that a nilpotent fusion category with Frobenius-Perron dimension congruent to 2 modulo 4 is solvable. Suppose $\mathcal C$ is nilpotent. Let us now prove the proposition by induction on $\fpdim(\mathcal C)$. 

If $\fpdim(\mathcal C) = 2$, $\mathcal C$ is pointed and then solvable. 
Now, suppose that $\fpdim(\mathcal C) > 2$. Since $\mathcal C$ is weakly group-theoretical, it is a $G$-extension of a fusion subcategory $\mathcal C_1$ for a non-trivial group $G$. Therefore,
\[\fpdim(\mathcal C) = |G|\fpdim(\mathcal C_1) \ \textrm{with} \ |G| > 1.\] 

If $|G|$ is odd, then $\fpdim(\mathcal C_1) < \fpdim(\mathcal C)$ is also congruent to 2 modulo 4 and must be solvable by the inductive hypothesis.
Then, $\mathcal C$ is solvable because it is a $G$-extension of $\mathcal C_1$, see \cite[Proposition 4.5]{ENO2}.

If $|G|$ is even, it is congruent to 2 modulo 4. Then, $\fpdim(\mathcal C_1)$ is odd and $\mathcal C_1$ must be solvable by \cite[Proposition 7.1]{NP}. Since $|G|$ is congruent to 2 modulo 4, from the generalized Feit-Thompson's Theorem, $G$ is solvable as well. So, it follows from \cite[Proposition 4.5]{ENO2} that $\mathcal C$ is solvable. 
\end{proof}

In particular, a direct consequence of this result is that if $\mathcal C$ is an MTC with dimension not divisible by 4 that is an extension of a pointed category then it must be solvable. 
\begin{corollary}\label{cadpointed}
Let $\mathcal C$ be an MTC of dimension not divisible by 4. If $\mathcal C_{\ad} \subseteq \mathcal C_{\pt}$ then $\mathcal C$ is solvable. So, if $\mathcal C$ is not solvable, all components of the universal grading containing invertible objects must contain at least one non-invertible simple object. 
\end{corollary} 
\begin{proof}
By hypothesis,  since $\Cad$ is pointed, we have that $\mathcal C$ is nilpotent. So, by definition, $\mathcal C$ is weakly group-theoretical and, as a consequence of Lemma~\ref{lemma: wgt dim cong 2 solv}  and \cite[Proposition 7.1]{NP}, $\mathcal{C}$ is solvable. 
The second comment follows directly from \cite[Proposition 3.17]{CGP}.
\end{proof}

\end{document}